\documentclass{article}

\usepackage{amsmath,amssymb,amsthm}
\usepackage{amsrefs}
\usepackage{bm} %bold 
\usepackage{bbm} %blackboardbold 
\usepackage{color} %color 
\usepackage{comment}
\usepackage{soul}
\usepackage{fullpage}

\usepackage{ulem}
\allowdisplaybreaks 

\theoremstyle{plain}
\newtheorem{theorem}{Theorem}[section]
\newtheorem{corollary}[theorem]{Corollary}
\newtheorem{proposition}[theorem]{Proposition}
\newtheorem{lemma}[theorem]{Lemma}
\newtheorem*{theorem*}{Theorem}
\newtheorem*{lemma*}{Lemma}
\newtheorem*{proposition*}{Proposition}
\newtheorem*{corollary*}{Corollary}

\theoremstyle{definition}

\newtheorem{example}[theorem]{Example}
\newtheorem{remark}[theorem]{Remark}

\newtheorem*{definition*}{Definition}
\newtheorem*{example*}{Example}
\newtheorem*{remark*}{Remark}

\title{On the multiple recurrence properties for disjoint systems}

\begin{document}

\author{Michihiro Hirayama\\
\small Department of Mathematics, University of Tsukuba,  Japan\\
\small e-mail: hirayama@math.tsukuba.ac.jp
\and 
Dong Han Kim\\
\small Department of Mathematics Education, Dongguk University-Seoul, Korea\\
\small e-mail: kim2010@dongguk.edu
\and 
Younghwan Son\\
\small Department of Mathematics,  POSTECH, Korea \\
\small e-mail: yhson@postech.ac.kr}

\maketitle

\begin{abstract}
We consider mutually disjoint family of measure preserving transformations $T_1, \cdots, T_k$ on a probability space $(X, \mathcal{B}, \mu)$. 
We obtain the multiple recurrence property of $T_1, \cdots, T_k$ and this result is utilized to derive multiple recurrence of Poincar\'e type in metric spaces. 
We also present multiple recurrence property of Khintchine type. 
Further, we study multiple ergodic averages of disjoint systems and we show that $T_1, \cdots, T_k$ are uniformly jointly ergodic if each $T_i$ is ergodic. 
\end{abstract}
\renewcommand{\thefootnote}{\fnsymbol{footnote}} 
\renewcommand{\thefootnote}{\arabic{footnote}} 
\section{Introduction}

One of the fundamental properties in the ergodic theory of dynamical systems is the recurrence property. 
For every probability measure preserving system $(X,\mathcal{B} ,\mu ,T)$, the Poincar\'e recurrence theorem states that for every $A\in \mathcal{B} $ with $\mu (A)>0$, the set 
\begin{equation} \label{poincare}
\left\{ n\in \mathbb{N} \colon \mu (A\cap T^{-n}(A))>0 \right\}
\end{equation}
is infinite. 
Furstenberg \cite{Furstenberg1977} proved his multiple recurrence result: for any set $A\in \mathcal{B} $ with $\mu (A)>0$ and a given positive integer $k$, the set 
\begin{equation} \label{furstenberg}
\left\{ n\in \mathbb{N} \colon \mu \left( A\cap T^{-n}(A)\cap T^{-2n}(A)\cap \dots \cap T^{-kn}(A)\right) >0 \right\}
\end{equation}
is infinite. 
Subsequently, Furstenberg and Katznelson \cite{Furstenberg-Katznelson1978} showed a commuting version of the multiple recurrence theorem. 
Let $T_1,\dots, T_k$ be commuting measure preserving transformations on $(X,\mathcal{B} ,\mu )$. 
Then it is proved in \cite{Furstenberg-Katznelson1978} that for any set $A\in \mathcal{B} $ with $\mu (A)>0$, there is some $c=c(A)>0$ such that 
\begin{equation} \label{furstenberg-katznelson}
\left\{ n\in \mathbb{N} \colon \mu \left( A\cap T_1^{-n}(A)\cap T_2^{-n}(A)\cap \dots \cap T_k^{-n}(A)\right) >c \right\}
\end{equation}
is syndetic. 
Here we recall that a subset $E\subset \mathbb{N} $ is said to be \emph{syndetic} if it has bounded gaps, that is, there is a positive integer $K$ such that $E\cap \{ n,n+1,\dots ,n+K-1\} \neq \emptyset $ for every $n\in \mathbb{N} $. 
Such sets are sometimes called \emph{relatively dense}. 

In this paper, we study several aspects of multiple recurrence for disjoint systems. 
%
%
%
%
%
%\subsection{Results} 

%We start with the following result which shows the set of the form in \eqref{furstenberg-katznelson} is syndetic for disjoint systems. 
%In other words, we prove the analogue of the Furstenberg-Katznelson multiple recurrence for disjoint systems. 
%
%We also show several related multiple recurrence results for disjoint systems as follows.
%
%
%
%
%
\subsection{Multiple recurrence of Khintchine type} 

The Khintchine recurrence theorem gives a quantitative improvement of the Poincar\'e recurrence \eqref{poincare}. 
Namely, Khintchine \cite{Khintchine1935} showed that for every $A\in \mathcal{B} $ with $\mu (A)>0$ and $\varepsilon >0$, the set 
\begin{equation*} \label{single_khintchine}
\left\{ n\in \mathbb{N} \colon \mu (A\cap T^{-n}(A))>\mu (A)^2-\varepsilon \right\}
\end{equation*}
is syndetic. 
It is therefore natural to ask whether a multiple recurrence result of Khintchine type can be established for \eqref{furstenberg} or \eqref{furstenberg-katznelson} in general, but it is not that straightforward. 

In \cite{Bergelson-Host-Kra2005}, Bergelson, Host and Kra proved the following results. 
Let $(X,\mathcal{B} ,\mu ,T)$ be an invertible ergodic system. 
Then for every $A\in \mathcal{B} $ with $\mu (A)>0$ and $\varepsilon >0$, the sets 
\begin{equation} \label{BHK_T-T^2}
\left\{ n\in \mathbb{Z} \colon \mu \left( A\cap T^{-n}(A)\cap T^{-2n}(A)\right) >\mu (A)^3-\varepsilon \right\}
\end{equation}
and 
\[
\left\{ n\in \mathbb{Z} \colon \mu \left( A\cap T^{-n}(A)\cap T^{-2n}(A)\cap T^{-3n}(A)\right) >\mu (A)^4-\varepsilon \right\}
\] 
are syndetic (\cite{Bergelson-Host-Kra2005}*{Theorem 1.2}), while 
\begin{itemize}
\item \cite{Bergelson-Host-Kra2005}*{Theorem 1.3} there exists an invertible ergodic system $(X,\mathcal{B} ,\mu ,T)$ such that for every $l\in \mathbb{N} $, there is $A=A_l\in \mathcal{B} $ with $\mu (A)>0$ such that 
\[
\mu \left( A\cap T^{-n}(A)\cap T^{-2n}(A)\cap T^{-3n}(A)\cap T^{-4n}(A)\right) \leq \mu (A)^l/2,
\] 
for every $n\in \mathbb{Z} \setminus \{ 0\} $,
\item \cite{Bergelson-Host-Kra2005}*{Theorem 2.1} there exists an invertible non-ergodic system $(X,\mathcal{B} ,\mu ,T)$ such that for every $l\in \mathbb{N} $, there is $A=A_l\in \mathcal{B} $ with $\mu (A)>0$ such that 
\[
\mu \left( A\cap T^{-n}(A)\cap T^{-2n}(A)\right) \leq \mu (A)^l/2
\]
for every $n\in \mathbb{Z} \setminus \{ 0\} $. 
\end{itemize}
It follows, rather surprisingly,  that one can not have a multiple analogue of the Khintchine recurrence result for the family $\{ T,\dots, T^k\} $ with $k\geq 4$ in general, and that ergodicity is a necessary condition for $k = 2$ and $3$ while it is not needed for the Khintchine (single) recurrence. 

For two commuting measure preserving transformations, Chu \cite{Chu2011} proved the following result. 
Let $T_1$ and $T_2$ be commuting measure preserving transformations on $(X,\mathcal{B} ,\mu )$. 
Assume that the system $(X,\mathcal{B} ,\mu ,T_1,T_2)$ is ergodic with respect to the measure preserving action of the group generated by $T_1$ and $T_2$. 
Then for every $A\in \mathcal{B} $ with $\mu (A)>0$ and $\varepsilon >0$, the sets 
\begin{equation} \label{C_T_1-T_2}
\left\{ n\in \mathbb{Z} \colon \mu \left( A\cap T_1^{-n}(A)\cap T_2^{-n}(A)\right) >\mu (A)^{4}-\varepsilon \right\}
\end{equation}
is syndetic (\cite{Chu2011}*{Theorem 1.1}),
while 
\begin{itemize}  
\item \cite{Chu2011}*{Theorem 1.2}
for any $c\in (0,1]$, there exists a commuting ergodic measure preserving system $(X,\mathcal{B} ,\mu ,T_1,T_2)$, and $A\in \mathcal{B} $ with $\mu (A)>0$ such that 
\[
\mu \left( A\cap T_1^{-n}(A)\cap T_2^{-n}(A)\right) \leq c\mu (A)^{3}
\]
for every $n\in \mathbb{Z} \setminus \{ 0\} $,
\end{itemize}

Notice that the exponent of $\mu (A)$ is three in \eqref{BHK_T-T^2} for $T$ and $T^2$, while it is four in \eqref{C_T_1-T_2} for two commuting measure preserving transformations in general, and it can not be reduced to three. 
In fact, Donoso and Sun \cite{Donoso-Sun2018}*{Theorem 1.2} determined the best exponent is four for two commuting measure preserving transformations such that the group generated by the two of them acts ergodically on $X$. 
They \cite{Donoso-Sun2018} also showed that there exists a commuting ergodic measure preserving system $(X,\mathcal{B} ,\mu ,T_1,T_2,T_3)$ such that for every $l\in \mathbb{N} $, there is $A=A_l\in \mathcal{B} $ with $\mu (A)>0$ such that 
\[
\left\{ n\in \mathbb{Z} \colon \mu \left( A\cap T_1^{-n}(A)\cap T_2^{-n}(A)\cap T_3^{-n}(A)\right) \leq \mu (A)^l\right\} =\mathbb{Z} \setminus \{ 0\} .
\]

% For ergodic disjoint systems, 
 We prove the following result which will give another aspect for a multiple recurrence of Khintchine type. 
 (See Section~\ref{preliminaries} for the notion of standard Borel probability space and disjointness.)

\begin{theorem} \label{mrkhintchine_disjoint_general}
Let $(X,\mathcal{B} ,\mu )$ be a standard Borel probability space. 
Given $k\in \mathbb{N}$, let $T_0,T_1,\dots ,T_k$ be measure preserving transformations on $X$. 
Suppose that $T_0,T_1,\dots ,T_k$ are mutually disjoint. 
Then for every $\varepsilon >0$ and every $A\in \mathcal{B} $ with $\mu (A)>0$, the set 
\[
\left\{ n\in \mathbb{N} \colon \mu \left( T_0^{-n}(A)\cap T_1^{-n}(A)\cap \dots \cap T_k^{-n}(A)\right) >\mu (A)^{k+1}-\varepsilon \right\} 
\]
is syndetic. 
\end{theorem}

Notice that the commutativity condition of the system is not necessary in Theorem \ref{mrkhintchine_disjoint_general}. 
For non-disjoint systems, Theorem \ref{mrkhintchine_disjoint_general} does not hold in general (see Examples \ref{failure_nondisjoint_commutative} and \ref{failure_ergodic_nondisjoint_noncommutative} below). 
%For non-disjoint systems, Theorem \ref{mrkhintchine_disjoint_general} does not hold in general (see Example \ref{failure_nondisjoint_commutative}).  
If $T_1,\dots ,T_k$ are ergodic and mutually disjoint, % measure preserving transformations on $X$. 
then ${\rm id}_X,T_1,\dots ,T_k$ are also mutually disjoint by Lemma \ref{disjointness_criterion_ergodic_id} below. 
%Applying Theorem \ref{mrkhintchine_disjoint_general} with $T_0={\rm id}_X$ yields the result. 
%\begin{theorem} \label{mrkhintchine_disjoint}
%Let $(X,\mathcal{B} ,\mu )$ be a standard Borel probability space, and
%Suppose that $T_1,\dots ,T_k$ are mutually disjoint. 
Therefore, for ergodic mutually disjoint $T_1,\dots ,T_k$, 
%for every $\varepsilon >0$ and every $A\in \mathcal{B} $ with $\mu (A)>0$, 
the set 
\begin{equation}\label{mrkhintchine_disjoint}
\left\{ n\in \mathbb{N} \colon \mu \left( A\cap T_1^{-n}(A)\cap T_2^{-n}(A)\cap \dots \cap T_k^{-n}(A)\right) >\mu (A)^{k+1}-\varepsilon \right\} 
\end{equation}
is syndetic.
%\end{theorem}

%In other words, under the assumption of the ergodicity, we can take $c=\mu (A)^{k+1}-\varepsilon $ in Theorem \ref{recurrence:new}.
%We can obtain even the following more general form. 

%The following result which shows the set of the form in \eqref{furstenberg-katznelson} is syndetic for disjoint systems. 
%In other words, 
After weakening the bound of \eqref{mrkhintchine_disjoint},  we show the analogue of the Furstenberg-Katznelson multiple recurrence for disjoint systems without assuming of the ergodicity.
In other words, the set of the form in \eqref{furstenberg-katznelson} is syndetic for disjoint systems. 

\begin{theorem}\label{recurrence:new}
Let $(X,\mathcal{B} ,\mu )$ be a standard Borel probability space, and let $T_1,\dots ,T_k$ be measure preserving transformations on $X$. 
Suppose that $T_1,\dots ,T_k$ are mutually disjoint. 
Then for every $A\in \mathcal{B} $ with $\mu (A)>0$, there is $c=c(A)>0$ such that the set 
\[
\left\{ n \in \mathbb{N} \colon \mu \left( A \cap T_1^{-n}(A) \cap T_2^{-n}(A)\cap \dots \cap T_k^{-n}(A)\right) > c \right\}
\]
is syndetic.
\end{theorem}

\begin{remark} 
For two commuting measure preserving transformations, the result analogous to Theorem \ref{mrkhintchine_disjoint_general} can be obtained easily.  
Indeed, let $T_0$ and $T_1$ be commuting invertible measure preserving transformations. 
Then 
\[
\mu \left( T_0^{-n}(A)\cap T_1^{-n}(A)\right) =\mu \left( A\cap (T_0^{-1}T_1)^{-n}(A)\right) ,
\]
holds for every $A\in \mathcal{B} $, and hence the set 
\[
\left\{ n\in \mathbb{N} \colon \mu \left( T_0^{-n}(A)\cap T_1^{-n}(A)\right) >\mu (A)^{2}-\varepsilon \right\} 
\]
is syndetic by the Khintchine (single) recurrence theorem for $T_0^{-1}T_1$. 
When $T_0$ and $T_1$ are non-invertible, by using the natural extension $\pi \colon (\widetilde{X} ,\widetilde{\mathcal{B}}, \widetilde{\mu }, \widetilde{T_0},\widetilde{T_1})\to (X,\mathcal{B} ,\mu ,T_0,T_1)$, one can show the same result. 
(See \cite{Furstenberg1981} for the natural extension of commuting systems.) 
While for three  commuting measure preserving transformations, the result analogous to Theorem \ref{mrkhintchine_disjoint_general} does not hold in general (see Example \ref{failure_nondisjoint_commutative}). 
\end{remark}
\subsection{$L^2$-convergence of multiple ergodic averages} 

Given measure preserving transformations $T_1,\dots ,T_k$ on a probability space $(X,\mathcal{B} ,\mu )$, the multiple recurrence \eqref{furstenberg-katznelson} relates to the $L^2$-convergence of the  following multiple ergodic averages 
\[
\frac{1}{N} \sum _{n=0}^{N-1}f_1\circ T_1^{n}\cdot f_2\circ T_2^{n}\cdot \cdots \cdot f_k\circ T_k^{n}=\frac{1}{N} \sum _{n=0}^{N-1}T_1^{n}f_1\cdot T_2^{n}f_2\cdot\cdots \cdot T_k^{n}f_k
\]
for $f_i\in L^\infty (X,\mu )$, $i\in \{ 1,\dots ,k\} $. 
Here and below, for a measure preserving transformation $S$ on $X$, we will denote an operator for functions $f$ on $X$ by $Sf=f\circ S$.
Note that the case $k=1$ is the von Neumann ergodic theorem, hence the $L^2$-limit is given by the orthogonal projection onto the invariant factor. 

Host and Kra \cite{Host-Kra2005}, and independently Ziegler \cite{Ziegler2007} proved the following result. 
Let $(X,\mathcal{B} ,\mu ,T)$ be an invertible probability measure preserving system. 
Then for every $f_i\in L^\infty (X,\mu )$, $i\in \{ 1,\dots ,k\} $, the limit
\begin{equation} \label{multiple_ergodic_average}
\lim _{N\to \infty }\frac{1}{N} \sum _{n=0}^{N-1}T^{n}f_1\cdot T^{2n}f_2\cdot \cdots \cdot T^{kn}f_k
\end{equation}
exists in $L^2(X,\mu )$. 
Subsequently, Tao \cite{Tao2008} proved the $L^2$-convergence for commuting transformations, or more precisely, for  commuting measure preserving transformations $T_1,\dots ,T_k$, the limit 
\begin{equation} \label{multiple_ergodic_average_commute}
\lim _{N\to \infty }\frac{1}{N} \sum _{n=0}^{N-1}T_1^{n}f_1\cdot T_2^{n}f_2\cdot\cdots \cdot T_k^{n}f_k
\end{equation}
exists in $L^2(X,\mu )$. 
Soon after, Towsner \cite{Towsner2009} gave a different proof by using nonstandard analysis, and Austin \cite{Austin2010} gave an ergodic proof.  
In fact, Austin \cite{Austin2010} proved more general results: the $L^2$-convergence of uniform averages over a F{\o}lner sequence.  
See also \cite{Host2009}. 
By way of contrast, Bergelson and Leibman \cite{Bergelson-Leibman2002} showed that the $L^2$-limit (7) of the (double) ergodic averages does not always exist if the group generated by $T_1$ and $T_2$ is not nilpotent, while the limit of ergodic averages exists if $T_1$ and $T_2$ generate a nilpotent group \cite{Bergelson-Leibman2002}*{Theorem A}. 
The study of $L^2$-convergence of more general multiple ergodic averages along the orbits of measure preserving action by a nilpotent group culminated in the remarkable result of Walsh \cite{Walsh2012}:
let $G$ be a nilpotent group of measure preserving transformations on $(X, \mathcal{B}, \mu)$. Then, for every $T_1, \dots, T_k \in G$, the following ergodic averages
\[ \frac{1}{N} \sum_{n=1}^N \prod_{j=1}^d \left( T_1^{p_{1,j} (n)} \cdots T_k^{p_{k,j} (n)} \right) f_j \]
converges in $L^2$ for every $f_1, \dots, f_d \in L^{\infty}$ and every set of integer valued polynomials $p_{i,j}(n)$. 

It is known that if $T$ is a weakly mixing transformation, then the $L^2$-limit \eqref{multiple_ergodic_average} is given by 
\[
\lim _{N - M \to \infty }\frac{1}{N-M} \sum _{n=M}^{N-1}T^{n}f_1\cdot T^{2n}f_2\cdot \cdots \cdot T^{kn}f_k=\prod _{i=1}^k\int _Xf_i\, d\mu . 
\]
See \cite{Furstenberg1977}, \cite{Berend-Bergelson1984}*{Corollary 3.1}, and \cite{Furstenberg-Katznelson-Ornstein1982}. 
Note, however, that the $L^2$-limit \eqref{multiple_ergodic_average} needs not be constant for general ergodic systems. 
Here, we obtain the following result. 

\begin{theorem} \label{multiple_uniform_average_disjoint}
Let $(X,\mathcal{B} ,\mu )$ be a standard Borel probability space. 
Suppose that $T_1,\dots ,T_k$ are ergodic measure preserving transformations on $X$ such that $T_1, T_2, \dots, T_k$ are mutually disjoint. 
Then for $f_i\in L^\infty (X,\mu )$, $i\in \{ 1,\dots ,k\} $, 
\[
\lim _{N-M \to \infty }\frac{1}{N-M} \sum _{n=M}^{N-1}T_1^{n}f_1\cdot T_2^{n}f_2\cdot\cdots \cdot T_k^{n}f_k=\prod _{i=1}^k\int _Xf_i\, d\mu 
\]
in $L^2(X,\mu )$. 
\end{theorem}

Notice again that the commutativity condition of the system is not necessary in Theorem \ref{multiple_uniform_average_disjoint}. 
Pointwise convergence of multiple ergodic averages for the case of two disjoint transformations ($k=2$), were studied via joining method \cites{Lesigne-Rittaud-delaRue2003, delaRue2012}. Especially, it was shown that if $T_1, T_2$ are ergodic and disjoint, then they are pointwise jointly ergodic. (See \cite{Berend1985}*{Theorem 2.2}. See also \cites{Furstenberg1967, delaRue2012}.)
\subsection{Multiple recurrence of Poincar\'e type in metric spaces} 

When the underlying space $X$ admits a metric, we have another refinement of the Poincar\'e's Recurrence theorem. 
Namely, let $(X,\mathcal{B} ,\mu )$ be a Borel probability space with a compatible metric $d$ such that $(X,d)$ is separable, and let $T$ be a measure preserving transformation on $X$. 
Then 
\[ 
\liminf _{n\to \infty }d\left( x,T^n(x)\right) =0
\]
for $\mu $-almost every $x\in X$. 
See \cite{Furstenberg1981}*{Theorem 3.3} for instance. 

A multiple analogue of this result for commuting systems is given as follows. 
Let $T_1,\dots ,T_k$ be commuting measure preserving transformation on $(X,\mathcal{B} ,\mu )$. 
Then 
\[ 
\liminf _{n\to \infty }{\rm diam}\left\{ x,T_1^n(x),\dots ,T_k^n(x)\right\} =0
\]
for $\mu $-almost every $x\in X$. 
See \cite{Hirayama2019}*{Propositions 1.2 and 6.1} for instance. 

The following result shows another multiple analogue for disjoint systems.  
\begin{theorem} \label{metric_recurrence:new} 
Let $(X,\mathcal{B} ,\mu )$ be a standard Borel probability space with a compatible metric $d$, and let $T_1,\dots ,T_k$ be measure preserving transformations on $X$. 
Suppose that $T_1,\dots ,T_k$ are mutually disjoint. 
Then
\[
\liminf _{n\to \infty } {\rm diam} \left\{x, T_1^n(x),\dots ,T_k^n(x)\right\} =0
\]
for $\mu $-almost every $x\in X$. 
\end{theorem}

In Section \ref{preliminaries}, we recall the notion of disjointness and gather the lemmas required in the proof. 
In Section \ref{poof}, we start by investigating the uniform averages of multiple correlation sequences, and then prove Theorems \ref{multiple_uniform_average_disjoint} and \ref{mrkhintchine_disjoint_general} in Section \ref{proof-1.6-1.3-1.2}. 
Finally, we prove Theorems \ref{recurrence:new} and \ref{metric_recurrence:new} in Section \ref{proof-1.1-1.5}.
\subsection{Examples} 

A typical example for which the results can be applied is the rotation on the circle $\mathbb{T} = \mathbb R / \mathbb Z $. 

\begin{example} 
Let $\omega _1,\dots ,\omega _k$ be irrational numbers so that they are rationally independent. 
Then the rotation $R_i(x) = x+ \omega_i$ on $\mathbb T$ is ergodic with respect to the Lebesgue measure on $\mathbb{T} $ for each $i\in \{ 1,\dots ,k\} $, and $R_i$ are disjoint. 
Therefore, by Theorem~\ref{metric_recurrence:new}, we have for almost every $x$
\[
\liminf _{n\to \infty } {\rm diam} \{ x,R_1^n(x), \dots ,R_k^n(x)\}  = 0.
\]
In fact, we can deduce a stronger result using Dirichlet's theorem, which states that
%By Dirichlet's thereom it is known that 
for any $N$ there exists $1 \le n \le N$ such that
\[
\max_{1 \le i \le k} \| n\omega _i\| < \frac{1}{N^{1/k}},
\] 
where $\| t \| = \min\{ | t-m | \, | \, m \in \mathbb Z\}$. 
Since
\[
{\rm diam} \{ x,R_1^n(x), \dots ,R_k^n(x)\} = {\rm diam} \{ x,x +n\omega_1, \dots ,x+n\omega_k\} \le 2\max_{1 \le i \le k} \| n\omega _i\|,
\]
we have for any $x$
\[
\liminf _{n\to \infty } n^{1/k} {\rm diam} \{ x,R_1^n(x), \dots ,R_k^n(x)\} 
%&= \liminf _{n\to \infty } n^{1/k} {\rm diam} \{ x,x +n\omega_1, \dots ,x+n\omega_k\} \\
 \le 2.
\]
For a related result of the multiple recurrence of irrational rotations, see \cite{Kim2009}.
\end{example}

%\begin{example} 
%Let $T\colon X\to X$ be a weakly mixing transformations with MSJ. 
%Then $T^k$ and $T^l$ are disjoint for every $k\neq l$. 
%See \cite{delJunco-Rudolph1987}*{Corollary 6.5}.
%\end{example}

An example of \cite{Furstenberg1981}*{p.40} shows that there is a non-disjoint and non-commutative system for which \eqref{mrkhintchine_disjoint} and Theorem \ref{recurrence:new} fail. 
It is outlined here for the sake of completeness. 

\begin{example} \label{failure_ergodic_nondisjoint_noncommutative}
Let $X=\{ 0,1\}^{\mathbb Z} $ and $(X,\mu ,S)$ be the two sided $(1/2,1/2)$-Bernoulli shift.
Define a ``flip" $\psi \colon X\ni x=(x_i)_{i\in \mathbb{Z} }\mapsto \psi (x)\in X$ by 
\[
(\psi (x))_i=\begin{cases} x_0, & i=0, \\ 1-x_i, & i\neq 0,\end{cases} 
\]
and then define $T\colon X\to X$ by $T=\psi ^{-1}\circ S\circ \psi $. 
We see $S\circ T\neq T\circ S$. 
Let $A\subset X$ be the cylinder set defined as 
\[
A=\{ x\in X\colon x_0=0\} .
\]
It is clear that $x\in A\cap S^{-n}(A)$ if and only if $x_0=x_n=0$. 
On the other hand, we see that any $x\in A\cap T^{-n}(A)$ needs to satisfy $x_0=(T^n(x))_0=0$. 
Note that $T^n=\psi ^{-1}\circ S^n\circ \psi $ and $\psi \circ \psi =\text{id} _X$. 
It follows that
\[
(T^n(x))_0=(\psi \circ S^n\circ \psi (x))_0=(S^n\circ \psi (x))_0=(\psi (x))_n=1-x_n
\]
for every $n\in \mathbb{Z} \setminus \{ 0\} $. 
Consequently, 
\[
A\cap S^{-n}(A)\cap T^{-n}(A)=\emptyset 
\]
for every $n\in \mathbb{Z} \setminus \{ 0\} $. 
Hence, for any $0 < \varepsilon < \mu (A)^{3}$
\[
\left\{ n\in \mathbb{N} \colon \mu \left( A\cap S^{-n}(A)\cap T^{-n}(A)\right) >\mu (A)^{3}-\varepsilon \right\} =\{ 0\} .
\]
Since both $S$ and $T$ have positive entropy, they cannot be disjoint (\cite{Furstenberg1967}*{Theorem I.1}). 
\end{example}

The following example, based on an example of \cite{Chu2011}*{Theorem 1.2}, shows that there is a non-disjoint, % ergodic
 and commutative system for which Theorem \ref{mrkhintchine_disjoint_general} fails. 

%\begin{example} \label{failure_nondisjoint_commutative}

%\end{example}

%Similarly, one can construct a non-disjoint and commutative example for which Theorem \ref{mrkhintchine_disjoint_general} fails. 
\begin{example} \label{failure_nondisjoint_commutative}
Let $\Sigma = \{ 0,1,2\}^{\mathbb Z} $ and $(\Sigma,\nu ,S)$ be the two sided $(1/3,1/3,1/3)$-Bernoulli shift.
Set $X=\Sigma\times \Sigma\times \Sigma$ endowed with a probability measure $\mu =\nu ^{\otimes 3}$. 
Define $T_0=\mathrm{id} _\Sigma \times S\times S$, $T_1=S\times \mathrm{id}_ \Sigma \times S$, and $T_2=S\times S\times \mathrm{id} _\Sigma $. 
Each $T_i$ $(i=0,1,2)$ preserves $\mu $ and $T_i\circ T_j=T_j\circ T_i$ ($j = 0, 1,2$). 
Let 
\[
A = \left\{ \left( (x_i)_{i \in \mathbb Z}, (y_i)_{i \in \mathbb Z}, (z_i)_{i \in \mathbb Z} \right) \in X\colon x_0, y_0, z_0 \text { are distinct} \right\} . 
\]
Since there are six distinct triples $(x_0, y_0, z_0)$ - $(0,1,2)$, $(0,2,1)$, $(1,0,2)$, $(1,2,0)$, $(2,0,1)$, $(2,1,0)$,
we have $ \mu (A)= 3!/3^3= 2/9$.
Note that 
\begin{equation} \label{distinct_triples}
\left( (x_i)_{i \in \mathbb Z}, (y_i)_{i \in \mathbb Z}, (z_i)_{i \in \mathbb Z}\right) \in T_0^{-n}(A)\cap T_1^{-n}(A)\cap T_2^{-n}(A) 
\end{equation}
if and only if  $(x_0, y_n, z_n)$, $(x_n, y_0, z_n)$, and $(x_n, y_n, z_0)$ are distinct triples.
Then, one can check that \eqref{distinct_triples} is achieved if and only if $(x_0, y_0, z_0)$ is a distinct triples and $x_0 = x_{n}$, $y_0 = y_{n}$, $z_0 = z_{n}$.
Therefore we have 
\[
\mu \left( T_0^{-n}(A)\cap T_1^{-n}(A)\cap T_2^{-n}(A) \right) = \frac{3!}{3^6} =  \frac{3}{4} \left( \frac{2}{9}\right)^3= \frac{3}{4} \mu (A)^3,
\]
hence we see
\[
\left\{ n\in \mathbb{N} \colon \mu \left( T_0^{-n}(A)\cap T_1^{-n}(A)\cap T_2^{-n}(A) \right) >\mu (A)^{3}-\varepsilon \right\} =\{ 0\} 
\]
for sufficiently small $\varepsilon >0$ and the set is not syndetic.
Since $T_0$, $T_1$, and $T_2$ have positive entropy, they cannot be disjoint (\cite{Furstenberg1967}*{Theorem I.1}). 

\end{example}
\section{Preliminaries} \label{preliminaries}

\subsection{Regular models}

We briefly recall the notion of conjugacy and isomorphism. 
See \cites{Furstenberg1981,Walters1982} for details. 
For a probability space $(X,\mathcal{B} ,\mu )$, let $\mathcal{B} _0=\{ [B]\colon B\in \mathcal{B} \} $, where $[B]$ is the equivalence class of $B$. 
For a measure preserving transformation $T$ on $(X,\mathcal{B} ,\mu )$, define $T_0^{-1}\colon \mathcal{B} _0\to \mathcal{B} _0$ by 
\[
T_0^{-1}([B])=[T^{-1}(B)]. 
\]

Suppose that $(X,\mathcal{B} ,\mu )$ is a standard Borel probability space, i.e., $X$ is a Polish space endowed with its Borel $\sigma $-algebra $\mathcal{B} $ and $\mu $ is a measure on it. 
Here and below, by a Polish space we mean a topological Polish space, that is, a separable completely metrizable topological space. 
Then the systems ${\bm X}=(X,\mathcal{B} ,\mu ,T)$ is a \emph{regular} system. 
That is, there exist a compact metrizable space $W$, a Borel probability measure $\nu $ on the Borel $\sigma $-algebra $\mathcal{W} $, a measurable transformation $S$ which preserves $\nu $,  and an invertible homomorphism $\varphi _0\colon \mathcal{W} _0\to \mathcal{B} _0$ such that 
\[
T_0^{-1}\circ \varphi _0=\varphi _0 \circ S_0^{-1}. 
\]
See \cite{Furstenberg1981}*{Proposition 5.3} for instance. 
The conjugacy transformation $\varphi _0$ can be chosen as an isomorphism up to measure zero. %taken isomorphism up to measure zero.  
Consequently, %Namely, 
there exists an invertible measure preserving transformation $\varphi \colon X\to W$ such that 
\[
\varphi \circ T=S\circ \varphi
\]
$\mu $-almost everywhere on $X$. 
See \cite{Furstenberg1981}*{Theorem 5.15} or \cite{Walters1982}*{Theorem 2.6}. 
Henceforth, we call such a systems ${\bm W}=(W,\mathcal{W} ,\nu ,S)$ a \emph{regular model} of ${\bm X}$. 
\subsection{Disjointness} \label{def_disjoint}

We recall the notion of disjointness in the sense of Furstenberg \cite{Furstenberg1967} briefly. 
See \cites{Rudolph1990, delaRue2012} for detail.
Let ${\bm X}=(X,\mathcal{B} _X, \mu ,T)$ and ${\bm Y}= (Y, \mathcal{B} _Y, \nu , S)$ be two probability measure preserving systems.
The \emph{joining} of the two systems ${\bm X}$ and ${\bm Y}$ is a probability measure on $X\times Y$ which is invariant under $T\times S$, and whose projections on $X$ and $Y$ are $\mu $ and $\nu $, respectively.
We denote the set of joinings of ${\bm X}$ and ${\bm Y}$ by $\mathcal{J} (T,S)$. 
The set $\mathcal{J} (T,S)$ is never empty since it contains the product measure $\mu \otimes \nu $. 
Note that $\mathcal{J} (T,S)$ is a convex set, where $s\lambda +(1-s)\rho $ is defined by $(s\lambda +(1-s)\rho )(E)=s\lambda (E)+(1-s)\rho (E)$ for $s\in [0,1]$. 
If ${\bm X}$ and ${\bm Y}$ are ergodic, then its extreme points are the ergodic joinings with respect to $T\times S$ (\cite{Rudolph1990}*{Theorem 6.3}). 
Two systems ${\bm X}$ and ${\bm Y}$ are said to be \emph{disjoint} if $\mu \otimes \nu $ is the unique joining of ${\bm X}$ and ${\bm Y}$. 

The definition of joining can be generalized to $k$-tuples of measure preserving systems ${\bm X}_i=(X_i,\mathcal{B} _i, \mu _i,T_i)$ and we denote by $\mathcal{J} (T_1,\dots ,T_k)$ the set of joinings. 
We call ${\bm X}_i$ or $T_i$ are \emph{mutually disjoint} if $\otimes _{i=1}^k\mu _i$ is the unique joining of ${\bm X}_i$'s.
If ${\bm Y}=(Y, \mathcal{B} _Y, \nu , S)$ is ergodic and $\lambda \in \mathcal{J} (T,S)$ is invariant under ${\rm id}_X\times S$, then $\lambda =\mu \otimes \nu $. 
See \cite{Rudolph1990}*{Lemma 6.14} for instance. 
The following lemma is an immediate consequence of this fact. 

\begin{lemma} \label{disjointness_criterion_ergodic_id}
Let $(X,\mathcal{B} , \mu )$ be a probability space, and let $T_1,\dots ,T_k$ be ergodic measure preserving transformations. 
If $T_1,\dots ,T_k$ are mutually disjoint, then so are ${\rm id}_X, T_1,\dots ,T_k$. 
\end{lemma}
\begin{proof}
Let $\lambda \in \mathcal{J} ({\rm id}_X,T_1,\dots ,T_k)$. 
By the assumption, we see that $T_1 \times \cdots \times T_k$ is ergodic with respect to $\mu^{\otimes k}$. 
Since $\lambda $ is invariant under ${\rm id}_X\times (T_1 \times \cdots \times T_k)$, we have $\lambda =\mu \otimes \mu^{\otimes k}$ by \cite{Rudolph1990}*{Lemma 6.14}. 
\end{proof}

Let $(W,\mathcal{W} ,\nu )$ be a Borel probability space, and let $S_1,\dots ,S_k$ be measure preserving transformations on $W$. 
Given integres $M<N$, define a probability measure $\lambda _{M,N}$ on $W^k$ by
\begin{equation} \label{unif_joining}
\lambda _{M,N}(A_1\times \cdots \times A_k)=\frac{1}{N-M} \sum _{n=M}^{N-1} \nu \left( S_1^{-n}(A_1)\cap \dots \cap S_k^{-n}(A_k)\right) 
\end{equation}
for $A_i\in \mathcal{W} $. 

\begin{lemma} \label{joinning}
Let $(W,\mathcal{W} ,\nu )$ be a compact metrizable Borel probability space, $S_1,\dots ,S_k$ be measure preserving transformations on $W$, and $\lambda _{M,N}$ a probability measure on $W^k$ defined by \eqref{unif_joining}. 
Then any weak*-limit measure of the sequence $\left( \lambda _{M,N} \right) _{M,N}$ as $N-M \to \infty$ belongs to $\mathcal{J} (S_1, \dots, S_k)$. 
\end{lemma}

\begin{proof}
Let $\lambda $ be a weak*-limit measure of $\left( \lambda _{M,N} \right) _{M,N}$ as $N-M \to \infty$. 
Let $p_i\colon W^k\ni (x_1,\dots , x_k) \mapsto x_i\in W$ be the canonical projection for each $i\in \{ 1,\dots ,k\} $. 
Then for every $i\in \{ 1,\dots ,k\} $, we see 
\begin{align*}
(p_i)_*\lambda _{M,N}(A_i)
&=\lambda _{M,N}(p_i^{-1}(A_i)) \\
&=\frac{1}{N-M} \sum _{n=M}^{N-1} \nu \left( S_1^{-n}(W)\cap \dots \cap S_i^{-n}(A_i)\cap \dots \cap S_k^{-n}(W)\right) \\
&=\frac{1}{N-M} \sum _{n=M}^{N-1} \nu \left( S_i^{-n}(A_i)\right) =\nu (A_i) ,
\end{align*}
and hence $(p_i)_*\lambda =\nu $. 

Next, we show $\lambda $ is invariant under $S_1 \times \dots \times S_k$. 
Note that 
\begin{align*}
\sum _{n=M}^{N-1} \nu \left( \cap _{i=1}^kS_i^{-(n+1)}(A_i)\right) -\sum _{n=M}^{N-1} \nu \left( \cap _{i=1}^kS_i^{-n}(A_i)\right) 
&=\nu \left( \cap _{i=1}^kS_i^{-(M+1)}(A_i)\right) -\nu \left( \cap _{i=1}^kS_i^{-M}(A_i)\right) \\
&\quad +\nu \left( \cap _{i=1}^kS_i^{-(M+2)}(A_i)\right) -\nu \left( \cap _{i=1}^kS_i^{-(M+1)}(A_i)\right) \\
&\quad +\cdots +\nu \left( \cap _{i=1}^kS_i^{-N}(A_i)\right) -\nu \left( \cap _{i=1}^kS_i^{-(N-1)}(A_i)\right) \\
&=\nu \left( \cap _{i=1}^kS_i^{-N}(A_i)\right) -\nu \left( \cap _{i=1}^kS_i^{-M}(A_i)\right) .
\end{align*}
Thus, given $M<N$, we have 
\begin{align*}
&\left| \lambda _{M,N}\left( S_1^{-1}(A_1)\times \cdots \times S_k^{-1}(A_k)\right) -\lambda _{M,N}\left( A_1\times \cdots \times A_k\right) \right| \\
&\quad =\frac{1}{N-M} \left| \sum _{n=M}^{N-1}\nu \left( \cap _{i=1}^kS_i^{-(n+1)}(A_i)\right) -\sum _{n=M}^{N-1} \nu \left( \cap _{i=1}^kS_i^{-n}(A_i)\right) \right| \\
&\quad =\frac{1}{N-M} \left| \nu \left( \cap _{i=1}^kS_i^{-N}(A_i)\right) -\nu \left( \cap _{i=1}^kS_i^{-M}(A_i)\right) \right| \leq \frac{2}{N-M} 
\end{align*}
for every $A_i\in \mathcal{W} $. 
Letting $N-M\to \infty $, it follows that $(S_1 \times \dots \times S_k)_*\lambda =\lambda $. 
\end{proof}

For a topological space $Y$, denote by $C(Y)$ the set of continuous functions on $Y$. 

\begin{lemma} \label{disjoint_unif_criterion}
Let $(W,\mathcal{W} ,\nu )$, $S_1,\dots ,S_k$, and $\lambda _{M,N}$ be as in Lemma \ref{joinning}. 
Suppose that $S_i$ are mutually disjoint. 
Then $\left( \lambda _{M,N} \right) _{M,N}$ converges to $\nu ^{\otimes k}$ with respect to the weak* topology as $N-M \to \infty$. 
\end{lemma}
\begin{proof}
It readily follows from Lemma \ref{joinning} since $\mathcal{J} (S_1, \dots, S_k)=\{ \nu ^{\otimes k}\} $. 
\end{proof}

For functions $\varphi _1,\dots ,\varphi _k$ on $W$, denote by $\varphi _1\otimes \dots \otimes \varphi _k$ the function on $W^k$ given by
\[
\varphi _1\otimes \dots \otimes \varphi _k(x_1,\dots ,x_k)=\varphi _1(x_1)\cdots \varphi _k(x_k). 
\]
Then note that the measure $\lambda _{M,N}$ defined in \eqref{unif_joining} can be characterized equivalently as
\begin{equation} \label{equiv_def_unif_joining}
\int _{W^k}\mathbbm{1} _{A_1}\otimes \dots \otimes \mathbbm{1} _{A_k}\, d\lambda _{M,N} =\frac{1}{N-M}\sum _{n=M}^{N-1} \int _WS_1^{n}\mathbbm{1} _{A_1}\cdot \cdots \cdot S_k^{n}\mathbbm{1} _{A_k}\, d\nu .
\end{equation}
Here and below, we denote by $\mathbbm{1} _{A}$ the indicator function of $A$, that is,
\[
\mathbbm{1} _A(x)=\begin{cases} 1,& x\in A,\\ 0, & x\not\in A. \end{cases}
\] 
\section{Multiple recurrence for disjoint systems} \label{poof}

\subsection{Uniform average of multiple correlation sequences}

In this subsection, we prove the following result. 

\begin{proposition} \label{unif_averag_mr_correlation_seq_disjoint_regular}
Let $(W,\mathcal{W} ,\nu )$ be a compact metrizable Borel probability space, $S_1,\dots ,S_k$ be measure preserving transformations on $W$. 
Suppose that $S_1,\dots ,S_k$ are mutually disjoint. 
Then for every $f_i\in L^\infty (W,\nu )$, $i\in \{ 1,\dots ,k\} $, we have
\[
\lim _{N - M \to \infty }\frac{1}{N-M} \sum _{n=M}^{N-1}\int _{W}\prod _{i=1}^kS_i^nf_i\, d\nu =\prod _{i=1}^k\int _{W}f_i\, d\nu . 
\]
\end{proposition}

\begin{proof}
Denote by $\| \cdot \| _p=\| \cdot \| _{L^p(W,\nu )}$ for every $p\in [1,\infty ]$. 
Let $\varepsilon \in (0,1)$ be given, and set $\Lambda = \max_{1 \leq i \leq k} \| f_i \|_\infty $. 
Note that $C(W)$ is dense in $L^{1} (W,\nu )$ since $W$ is compact Hausdorff. 
Since $L^\infty (W,\nu )\subset L^{1} (W,\nu )$, we can find $\varphi _i\in C(W)$ such that  for each $i=1,\dots ,k$
\[
\| f_i-\varphi _i\| _{1}<\varepsilon \quad \text{and} \quad \| \varphi _i\| _\infty \leq \| f_i\| _\infty \le \Lambda.
\]
(See e.g. \cite{Rudin1987}*{Lusin's Theorem 2.24 and Theorem 3.14}.) 
For notational simplicity, we write 
\[
\bar{f} =f_1\otimes \dots \otimes f_k \quad \text{and} \quad \bar{\varphi } =\varphi _1\otimes \dots \otimes \varphi _k.
\]
Note that $\bar{f} \in L^\infty (W^{k},\nu ^{\otimes k})$ and $\bar{\varphi } \in C(W^{k})$. 

By the triangle inequality, we have
\begin{align}
&\left| \frac{1}{N-M} \sum _{n=M}^{N-1}\int _{W}\prod _{i=1}^kS_i^nf_i\, d\nu -\int _{W^{k}}\bar{f} \, d\nu ^{\otimes k}\right| \nonumber \\ 
%&\leq \frac{1}{N-M} \sum _{n=M}^{N-1}\left| \int _{W}\bar{f} \left( S_1^n(x),\dots ,S_k^n(x)\right) \, d\nu (x)-\int _{W}\bar{\varphi } \left( S_1^n(x),\dots ,S_k^n(x)\right) \, d\nu (x)\right| \nonumber \\
%&\quad +\left| \frac{1}{N-M} \sum _{n=M}^{N-1}\int _{W}\bar{\varphi} \left( S_1^n(x),\dots ,S_k^n(x)\right) \, d\nu (x)-\int _{W^{k}}\bar{\varphi } \, d\nu ^{\otimes k}\right| \nonumber \\
%&\quad +\left| \int _{W^{k}}\bar{\varphi } \, d\nu ^{\otimes k}-\int _{W^{k}}\bar{f} \, d\nu ^{\otimes k}\right| \nonumber \\
&\leq \frac{1}{N-M} \sum _{n=M}^{N-1}\int _{W}\left| \bar{f} \left( S_1^n(x),\dots ,S_k^n(x)\right) -\bar{\varphi } \left( S_1^n(x),\dots ,S_k^n(x)\right) \right| \, d\nu (x) \label{sum} \\
&\quad +\left| \frac{1}{N-M} \sum _{n=M}^{N-1}\int _{W}\bar{\varphi} \left( S_1^n(x),\dots ,S_k^n(x)\right) \, d\nu (x)-\int _{W^{k}}\bar{\varphi } \, d\nu ^{\otimes k}\right| \label{disjoint} \\
&\quad +\int _{W^{k}}\left| \bar{\varphi }-\bar{f} \right| \, d\nu ^{\otimes k}. \label{inetgral}
\end{align}
Henceforth, we estimate these three terms \eqref{sum}, \eqref{disjoint} and \eqref{inetgral}. 
We begin with \eqref{inetgral}.  
Since
\begin{equation}\label{expand}\begin{split}
f_1(x_1)\cdots f_k(x_k)-\varphi _1(x_1)\cdots \varphi _k(x_k)
&= f_1(x_1)\cdots f_k(x_k)-\varphi _1(x_1)f_2(x_2)\cdots f_k(x_k) \\
&\quad + \varphi _1(x_1)f_2(x_2)\cdots f_k(x_k)-\varphi _1(x_1) \varphi_2(x_2) f_3 (x_3) \cdots f_k(x_k) \\
&\quad + \cdots + \varphi _1(x_1) \cdots \varphi_{k-1}(x_{k-1}) f_k(x_k) -\varphi _1(x_1) \varphi_2(x_2)  \cdots \varphi _k(x_k),
\end{split}\end{equation}
we have
\begin{align}
\left\| \bar f - \bar \varphi \right\|_{L^1(W^{k},\nu ^{\otimes k})} 
&\leq \| f_1 -\varphi _1 \|_1 \| f_2\|_\infty \cdots \| f_k \|_\infty +  \| f_2 -\varphi _2\|_1 \| \varphi_1\|_\infty \| f_3 \|_\infty \cdots \| f_k \|_\infty \nonumber \\
&\quad + \cdots + \| f_k -\varphi _k \|_1 \| \varphi _1\|_\infty \cdots \| \varphi_{k-1} \|_\infty \nonumber \\
&\leq k\Lambda ^{k-1} \varepsilon . \label{caluculus1}
\end{align}

Using the same argument as above, one can estimate \eqref{sum}. 
For a given $n \in [M,N)\cap \mathbb N$, replacing $x_i$ by $S_i^n (x)$ for $i\in \{ 1,\dots ,k\} $ in \eqref{expand}, we have
\begin{align*}
&\left| \bar\varphi \left( S_1^n(x),\dots ,S_k^n(x)\right) - \bar f \left( S_1^n(x),\dots ,S_k^n(x)\right) \right| \\
&\leq | f_1 (S_1^n(x)) -\varphi _1(S_1^n(x)) | \cdot \| f_2 \|_\infty \cdots \| f_k \|_\infty \\
&\quad +  | f_2 (S_2^n(x)) -\varphi _2(S_2^n(x)) | \cdot \| \varphi _1\|_\infty \| f_3 \|_\infty \cdots \| f_k \|_\infty \\
&\quad + \cdots + | f_k(S_k^n(x)) -\varphi _k(S_k^n(x)) | \cdot \| \varphi _1\|_\infty \cdots \| \varphi_{k-1} \|_\infty\\
&\leq \Lambda ^{k-1} \left( |f_1 (S_1^n(x)) -\varphi _1(S_1^n(x))|+\cdots +|f_k(S_k^n(x)) -\varphi _k(S_k^n(x))| \right) .
\end{align*}
By integrating, we have
\begin{equation} \label{caluculus2}
\int _{W}\left| \bar{f} (S_1^n(x),\dots ,S_k^n(x))-\bar{\varphi } (S_1^n(x),\dots ,S_k^n(x))\right| \, d\nu (x)\leq k\Lambda ^{k-1} \varepsilon .
\end{equation}

To estimate \eqref{disjoint}, recall the measure $\lambda _{M,N}$ defined in \eqref{unif_joining}. 
Notice that one has
\[
\int _{W^{k}}\bar{\varphi } \, d\lambda _{M,N}=\frac{1}{N-M} \sum _{n=M}^{N-1}\int _{W}\varphi _1\circ S_1^n\cdot \cdots \cdot \varphi _k\circ S_{k}^n\, d\nu 
\]
by using \eqref{equiv_def_unif_joining}. 
Then, by Lemma \ref{disjoint_unif_criterion}, we have 
\begin{equation} \label{caluculus3}
\lim _{N-M\to \infty }\frac{1}{N-M} \sum _{n=M}^{N-1}\int _{W}\bar{\varphi }(S_1^n(x),\dots ,S_{k}^n(x))\, d\nu (x)=\int _{W^{k}}\bar{\varphi } \, d\nu ^{\otimes k}.
\end{equation}

By \eqref{caluculus1}, \eqref{caluculus2} and \eqref{caluculus3}, we have 
\[
\left| \frac{1}{N-M} \sum _{n=M}^{N-1}\int _{W}\prod _{i=1}^kS_i^nf_i\, d\nu -\int _{W^{k}}\bar{f}\, d\nu ^{\otimes k}\right| \leq 2k\Lambda ^{k-1}\varepsilon +o_{N-M}(1),
\]
as $N-M\to \infty $. 
Since
\[
\int _{W^{k}}\bar{f}\, d\nu ^{\otimes k}=\prod _{i=1}^k\int _{W}f_i\, d\nu ,
\] 
we complete the proof of Proposition \ref{unif_averag_mr_correlation_seq_disjoint_regular}. 
\end{proof}

\begin{theorem} \label{unif_averag_mr_correlation_seq_disjoint}
Let $(X,\mathcal{B} ,\mu )$ be a standard Borel probability space.
Given $k\in \mathbb{N}$, let $T_0, T_1,\dots ,T_k$ be measure preserving transformations on $X$ such that $T_0, T_1,\dots ,T_k$ are mutually disjoint. 
Then for every $f_i\in L^\infty (X,\mu )$, $i\in \{ 0,1,\dots ,k\} $, 
\[
\lim _{N-M\to \infty }\frac{1}{N-M} \sum _{n=M}^{N-1}\int _{X} \prod _{i=0}^kT_i^nf_i\, d\mu =\prod _{i=0}^k\int _{X}f_i\, d\mu . 
\]
\end{theorem}

\begin{proof}
Let ${\bm W}=(W,\mathcal{W} ,\nu ,\{ S_i\} )$ be a regular model of $(X,\mathcal{B} ,\mu ,\{ T_i\} )$ via an isomorphism $\varphi \colon X\to W$. 
Namely, for every $i\in \{ 0,1,\dots ,k\} $, 
\[
\varphi \circ T_i=S_i\circ \varphi
\]
$\mu $-almost everywhere on $X$. 
Since $T_i$ are disjoint, so are $S_i$.  

Let $f_i\in L^\infty (X,\mu )$, $i\in \{ 0,1,\dots ,k\} $. 
Then we have
\[
\int _{X} \prod _{i=0}^kf_i(T_i^n(x))\, d\mu (x)=\int _{W} \prod _{i=0}^kf_i(T_i^n(\varphi ^{-1}(w)))\, d\nu (w) =\int _{W} \prod _{i=0}^k(f_i\circ \varphi ^{-1})\circ S_i^n(w)\, d\nu (w).
\]
Note that $f_i\circ \varphi ^{-1}\in L^\infty (W,\nu )$ and 
\[
\int _Wf_i\circ \varphi ^{-1}\, d\nu =\int _Xf_i\, d\mu
\]
for every $i\in \{ 0,1,\dots ,k\} $. 
Applying Proposition \ref{unif_averag_mr_correlation_seq_disjoint_regular}, we obtain
\begin{align*}
\frac{1}{N-M} \sum _{n=M}^{N-1}\int _{X} \prod _{i=0}^kf_i(T_i^n(x))\, d\mu (x) 
&=\frac{1}{N-M} \sum _{n=M}^{N-1}\int _{W} \prod _{i=0}^k(f_i\circ \varphi ^{-1})\circ S_i^n(w)\, d\nu (w) \\
&\to \prod _{i=1}^k\int _{W}f_i\circ \varphi ^{-1}\, d\nu =\prod _{i=1}^k\int _Xf_i\, d\mu
\end{align*}
as $N-M\to \infty $.
\end{proof}

\subsection{Proof of Theorems \ref{multiple_uniform_average_disjoint} and \ref{mrkhintchine_disjoint_general}}
\label{proof-1.6-1.3-1.2}
\subsubsection{Uniform joint ergodicity for disjoint systems} \label{proof_th1.3}

First, we recall the notion of joint ergodicity. 
See \cites{Berend-Bergelson1984, Berend1985, Berend-Bergelson1986}. 
Let $T_1,\dots ,T_k$ be measure preserving transformations on a probability space $(X,\mathcal{B} ,\mu )$. 
The system $(X,\mathcal{B} ,\mu ,T_1,\dots ,T_k)$ is called $L^2$-\emph{jointly ergodic} if 
\begin{equation} \label{def_joint_ergodic}
\frac{1}{N} \sum _{n=0}^{N-1}\prod _{i=1}^kT_i^nf_i\to \prod _{i=1}^k\int _Xf_i\, d\mu ,
\end{equation}
as $N\to \infty $ in $L^2$-norm for every $f_1,\dots ,f_k\in L^\infty (X,\mu )$. 
It is called $L^2$-\emph{weak jointly ergodic} if the convergence \eqref{def_joint_ergodic} takes place in weak $L^2$. 
Similarly, the system $(X,\mathcal{B} ,\mu ,T_1,\dots ,T_k)$ is called $L^2$-\emph{uniformly jointly ergodic} if 
\begin{equation} \label{def_uniform_joint_ergodic}
\frac{1}{N-M} \sum _{n=M}^{N-1}\prod _{i=1}^kT_i^nf_i\to \prod _{i=1}^k\int _Xf_i\, d\mu ,
\end{equation}
as $N-M\to \infty $ in $L^2$-norm for every $f_1,\dots ,f_k\in L^\infty (X,\mu )$. 
It is called $L^2$-\emph{weak uniformly jointly ergodic} if the convergence \eqref{def_uniform_joint_ergodic} takes place in weak $L^2$. 

In \cite{Berend-Bergelson1986}, Berend and Bergelson showed the following. 

\begin{theorem}[\cite{Berend-Bergelson1986}*{Theorem 2.1 and Remark 2.1}] \label{Berend-Bergelson_criterion}
Let $T_1,\dots ,T_k$ be measure preserving transformations on a probability space $(X,\mathcal{B} ,\mu )$. 
Then the following conditions are equivalent.
\begin{enumerate}
\item The system is $L^2$-uniformly jointly ergodic.
\item The system is $L^2$-weak uniformly jointly ergodic. 
\item \begin{enumerate}
\item The product systems $T_1\times \dots \times T_k$ is ergodic with respect to $\mu ^{\otimes k}$. 
\item For every $f_1,\dots ,f_k\in L^\infty (X,\mu )$, 
\[
\lim _{N-M\to \infty }\frac{1}{N-M} \sum _{n=M}^{N-1}\int _X\prod _{i=1}^kT_i^nf_i\, d\mu =\prod _{i=1}^k\int _Xf_i\, d\mu .
\]
\end{enumerate}
\end{enumerate}
\end{theorem}

We can prove Theorem \ref{multiple_uniform_average_disjoint} by applying Theorem \ref{Berend-Bergelson_criterion}. 
 
\begin{proof}[Proof of Theorem \ref{multiple_uniform_average_disjoint}]
Since each $T_i$ is ergodic and $T_1, \dots, T_k$ are mutually disjoint, it follows that $T_1\times \dots \times T_k$ is ergodic with respect to $\mu ^{\otimes k}$, and hence 3-(a) of Theorem \ref{Berend-Bergelson_criterion} is verified. 
Theorem \ref{unif_averag_mr_correlation_seq_disjoint} without $T_0$ implies 3-(b) of Theorem \ref{Berend-Bergelson_criterion}. 
Since the conclusion of Theorem~\ref{multiple_uniform_average_disjoint} is the same with 1 of Theorem ~\ref{Berend-Bergelson_criterion}, we complete the proof.
\end{proof}

The notions $L^2$-joint ergodicity and $L^2$-weak joint ergodicity are also equivalent, see \cite{Berend-Bergelson1986}*{Theorem  2.1}.  
However, the $L^2$-joint ergodicity does not imply the $L^2$-uniform joint ergodicity in general, see \cite{Berend-Bergelson1986}*{Example 3.1}. 
For commuting systems, Berend and Bergelson \cites{Berend1985, Berend-Bergelson1984} showed that all the four notions defined above are equivalent. 
The following result shows that disjoint systems possess all the four notions even if the systems are non-commutative. 

\begin{corollary} \label{dijoint_joint_ergodic} 
Let $(X,\mathcal{B} ,\mu)$ and $T_1,\dots ,T_k$ be as in Theorem \ref{multiple_uniform_average_disjoint}. 
Then it is $L^2$-uniformly jointly ergodic. 
\end{corollary}
\subsubsection{Multiple Khintchin recurrence} \label{proof_th1.1}

\begin{proposition} \label{functional_kmr_disjoint_general}
Let $(X,\mathcal{B} ,\mu )$ be a standard Borel probability space. 
Given $k\in \mathbb{N}$, suppose that $T_0,T_1,\dots ,T_k$ be measure preserving transformations on $X$ such that $T_0, T_1, \dots, T_k$ are mutually disjoint. 
Then for every $\varepsilon >0$ and every $f_i\in L^\infty (X,\mu )$ with $f_i\geq 0$, $i\in \{0,1,\dots ,k\} $, the set 
\[
\left\{ n\in \mathbb{N} \colon \int _X \prod _{i=0}^kT_i^nf_i\, d\mu >\prod _{i=0}^k\int _{X}f_i\, d\mu -\varepsilon \right\} 
\]
is syndetic. 
\end{proposition}
\begin{proof}
We may and do assume that %$f_i\neq 0$ for $\mu $-almost everywhere and for every $i\in \{ 0,1,\dots ,k\} $. 
it is not the case that for almost all $x$, $f_i(x) = 0$ for any $i\in \{ 0,1,\dots ,k\} $. 
Suppose that the set is not syndetic for some $\varepsilon _0>0$ and $f_i\in L^\infty (X,\mu )$ with $f_i\geq 0$, $i\in \{ 0,1,\dots ,k\} $. 
Then there exists a sequence of intervals $[M_j,N_j)$ with $N_j-M_j\to \infty $ as $j\to \infty $ such that for every $n\in [M_j,N_j)\cap \mathbb{N}$ and $j\in \mathbb{N} $, it holds that 
\[
\int _{X} \prod _{i=0}^kT_i^nf_i\, d\mu \leq \prod _{i=0}^k\int _{X}f_i\, d\mu -\varepsilon _0.
\]
It follows that 
\[
\frac{1}{N_j-M_j} \sum _{n=M_j}^{N_j-1} \left( \prod _{i=0}^k\int _{X}f_i\, d\mu -\int _{X} \prod _{i=0}^kT_i^nf_i\, d\mu \right) \geq \varepsilon
_0\]
for every $j\in \mathbb{N} $. 

On the other hand, by Theorem~\ref{unif_averag_mr_correlation_seq_disjoint}, we have 
\[
\lim _{N-M\to \infty } \frac{1}{N-M} \sum _{n=M}^{N-1} \int _{X}\prod _{i=0}^kT_i^nf_i\, d\mu =\prod _{i=0}^k\int _{X}f_i\, d\mu .
\]
This gives a contradiction, and hence the result follows.
\end{proof}

\begin{proof}[Proof of Theorem \ref{mrkhintchine_disjoint_general}]
Letting $f_i=\mathbbm{1} _A$ for every $i\in \{ 0,1,\dots ,k\} $ in Proposition \ref{functional_kmr_disjoint_general}, the result follows. 
\end{proof}

%\begin{proof}[Proof of Theorem \ref{mrkhintchine_disjoint}]
%By Lemma \ref{disjointness_criterion_ergodic_id}, we see that ${\rm id}_X,T_1,\dots ,T_k$ are mutually disjoint. 
%Applying Theorem \ref{mrkhintchine_disjoint_general} with $T_0={\rm id}_X$ yields the result. 
%\end{proof}

\subsection{Proof of Theorems \ref{recurrence:new} and \ref{metric_recurrence:new}}
\label{proof-1.1-1.5}

Before proving Theorems \ref{recurrence:new} and \ref{metric_recurrence:new}, let us introduce the following Hilbert space splitting.

\subsubsection{Hilbert space splitting: the Jacobs - de Leeuw - Glicksberg decomposition}
Let $U$ be an isometry on a Hilbert space $\mathcal{H}$. 
An element $f \in \mathcal{H}$ is {\em compact} if 
\[ \left\{ U^n f\colon n\in \mathbb{N} \cup \{0\} \right\} \]
is a pre-compact subset of $\mathcal{H}$. 
Then one has the following splitting theorem. (See, for example, \cite{Krengel1985}.)

\begin{theorem} Let $U$ be an isometry on a Hilbert space $\mathcal{H}$. 
Then
\[\mathcal{H} = \mathcal{H}_{\rm c} \oplus \mathcal{H}_{\rm wm},\]
where 
\[ \mathcal{H}_{\rm c}= \{f \in \mathcal{H} \colon f \text{ is compact } \} \]
and 
\[ 
\mathcal{H}_{\rm wm} = \left\{ f \in \mathcal{H} \colon \frac{1}{N} \sum_{n=0}^{N-1} | \langle U^n f, g  \rangle | \xrightarrow[N\to \infty ]{} 0\text{ for all } g \in \mathcal{H}\right\} .
\]
\end{theorem}

\begin{remark}
\label{rem:add} One can check 
\[ \mathcal{H}_{\rm c} = \overline{\text{span} \{f \in \mathcal{H} \colon Uf = \lambda f \text{ for some } \lambda \in  \mathbb{C} \text{ with } |\lambda|=1  \} }. \]
\end{remark}

Now we will prove the following lemma, which will be used later. 
\begin{lemma}
\label{lem:add}
Let $T$ be a measure preserving transformation on a probability space $(X, \mathcal{B}, \mu)$. 
We also regard $T$ as an isometry on $\mathcal{H}= L^2(X,\mu )$ by $Tf = f \circ T$.
For a measurable set $A \in \mathcal{B}$, write 
\[\mathbbm{1} _A = f + g,\]
where $f \in \mathcal{H}_{\rm c}$ and $g \in \mathcal{H}_{\rm wm}$. 
Then
\begin{enumerate}
\item $ 0 \leq f(x) \leq 1 $ for $\mu$-almost every $x\in X$.
\item $ \int \mathbbm{1} _{A}\cdot f \, d\mu \geq \mu (A)^2$. 
\item $f(x) > 0$ for $\mu$-almost every $x \in A$.
\end{enumerate}
\end{lemma}

\begin{proof} 
Without loss of generality, we assume that $\mu (A) > 0$. 
\begin{enumerate}
\item Suppose that $f = f_1 + i f_2$. 
Note that $f_1, f_2 \in \mathcal{H}_c$ since $f_1 = \frac{f+ \overline{f}}{2}$ and $f_2 = \frac{f - \overline{f}}{2i}$. 
Moreover we have that $\| f_1 - \mathbbm{1} _{A} \|_2 \leq \| f- \mathbbm{1} _A\|_2$. 
The fact that $f$ is an orthogonal projection of $\mathbbm{1} _A$ implies that $f= f_1$, so $f$ is real-valued. 
 
Now let $\tilde{f} = \max (\min(f,1), 0)$. 
Again one can check that $\tilde{f} \in \mathcal{H}_c$ and $\| \tilde{f} - \mathbbm{1} _{A} \|_2 \leq \| f- \mathbbm{1} _A\|_2 $.  
So $f = \tilde{f}$.

\item Let $P\colon \mathcal{H} \rightarrow \mathcal{H}_{\rm c}$ be the orthogonal projection. 
Then $P^2 = P$ and $P^*$, the adjoint of $P$, is $P$. 
Thus one has
\[
\langle P\mathbbm{1} _A, 1 \rangle = \langle P^2 \mathbbm{1} _A, 1 \rangle = \langle P\mathbbm{1} _A, P1 \rangle = \langle \mathbbm{1} _A, 1 \rangle = \mu (A).
\]
Also note that
\[ 
\int \mathbbm{1} _A\cdot f \, d\mu = \langle \mathbbm{1} _A, P\mathbbm{1} _A \rangle = \langle P\mathbbm{1} _A, P\mathbbm{1} _A \rangle  \geq  \langle P\mathbbm{1} _A,1 {\rangle}^2 = \mu(A)^2.
\]

\item Let $B = \{x \in A\colon f(x) = 0\}$. 
Suppose that $\mu(B)>0$.
Write 
\[\mathbbm{1} _{B} = F + G,\]
where $F \in \mathcal{H}_{\rm c}$ and $G \in \mathcal{H}_{\rm wm}$.
We will show that $F(x) = 0$ for $\mu$-almost every $x \in B$ and this leads to a contradiction: 
\[
0 = \int \mathbbm{1} _{B}\cdot F\, d\mu \geq \mu(B)^2 >0.
\]
Let us show that $F(x) = 0$ on $B$. 
First, note that $F(x) \geq 0$ for $\mu$-almost every $x\in X$, since $P \mathbbm{1}_B = F$.
Now write $\mathbbm{1} _{A \setminus B} = (f- F) + (g-G)$. 
Note that $f- F \in \mathcal{H}_{\rm c}$ and $g-G \in \mathcal{H}_{\rm wm}$. 
Thus $P \mathbbm{1}_{A \setminus B} = f- F$, so $f(x) - F(x) \geq 0$  for $\mu$-almost every $x\in X$. Then for $x \in B$,  $f(x) = 0$, so we have that $F(x) \leq 0$ for $\mu$-almost every $x \in B$. 
\end{enumerate}
Lemma \ref{lem:add} is proved. 
\end{proof}
\subsubsection{Multiple recurrence}

Let us consider the Jacobs-de Leeuw-Glicksberg decomposition of $\mathcal{H}= L^2(X,\mu )$ for each $T_i$ $(1 \leq i \leq k)$, that is, 
\[
\mathcal{H} = \mathcal{H}_{\rm c}^{i} \oplus \mathcal{H}_{\rm wm}^i.
\]
Write $\mathbbm{1} _{A} = f_i + g_i$ for $1 \leq i \leq k$, where $f_i \in \mathcal{H}_{\rm c}^i$ and $g_i \in \mathcal{H}_{\rm wm}^i$. 

Before proving Theorem \ref{recurrence:new}, let us first show the following lemma.
\begin{lemma}
\label{avg:wm}
If $u_1, \dots, u_k \in L^{\infty}(X)$ are real-valued functions and one of $u_i \in \mathcal{H}_{\rm wm}^i$, then 
\[ \lim_{N - M \rightarrow \infty} \frac{1}{N - M} \sum_{n=M}^{N -1} \prod_{i=1}^k T_i^n u_i = 0\]
in $L^2(X,\mu )$.
\end{lemma}

We need the following uniform version of the van der Corput lemma. (See lemma on p. 446 in \cite{Bergelson-Leibman2002})
\begin{lemma}[Van der Corput trick] \label{vdC}
Let $(x_n)_{n \in \mathbb{N}}$ be a bounded sequence in a Hilbert space $\mathcal{H}$. 
Then 
\[
\limsup_{N - M \rightarrow \infty} \left\| \frac{1}{N - M} \sum_{n=M}^{N-1} x_n \right\| ^2 \leq  \limsup_{H \rightarrow \infty} \frac{1}{H} \sum_{h=1}^H \limsup_{N - M \rightarrow \infty} \frac{1}{N -M } \sum_{n=M}^{N-1} \mathrm{Re} \langle x_{n+h}, x_n \rangle .
\]
\end{lemma}

\begin{proof}[Proof of Lemma \ref{avg:wm}]
Let $x_n =  \prod_{i=1}^k T_i^n u_i $. 
Then 
\[ \langle x_{n+h}, x_n \rangle = \int \prod_{i=1}^k T_i^n (u_i\cdot T_i^h u_i)\, d\mu . \]
Use Theorem \ref{unif_averag_mr_correlation_seq_disjoint} to obtain that 
\[ \lim_{N - M \rightarrow \infty} \frac{1}{N - M} \sum_{n=M}^{N-1} \langle x_{n+h}, x_n \rangle = \prod_{i=1}^k \int u_i\cdot T_i^h u_i\, d\mu .\]
If $u_i \in \mathcal{H}_{\rm wm}^i$, then 
\[ \lim_{H \rightarrow \infty} \frac{1}{H} \sum_{h=1}^H  \left| \int u_i\cdot T_i^h u_i\, d\mu \right| = 0, \]
so we have that
\[ \lim_{H \rightarrow \infty} \frac{1}{H} \sum_{h=1}^H  \left| \prod_{i=1}^k \int u_i\cdot T_i^h u_i\, d\mu  \right| = 0.\]
By Lemma \ref{vdC}, the result follows.
\end{proof}

\begin{proof}[Proof of Theorem \ref{recurrence:new}]
For each $1 \leq i \leq k$, we have $\mathbbm{1} _{A} = f_i + g_i$ as above. 
Then 
\[ \mu(A \cap T_1^{-n}A \cap \cdots \cap T_k^{-n}A) =  \left\langle \mathbbm{1} _A, \prod_{i=1}^{k} \left( T_i^n f_i + T_i^n g_i\right) \right\rangle _{L^2(X,\mu)}\]
If we multiply out the product of the right-hand side, there will be $2^{k}$ terms of the form
\[ T_1^n u_1(x)\cdot \, \cdots \, \cdot T_k^n u_k(x),\]
 where $u_i = f_i$ or $g_i$. 
Among these expressions, except the case that $u_i = f_i$ for all $i\in \{ 1,\dots ,k\}$, Lemma \ref{avg:wm} implies that 
\[ \lim_{N-M \rightarrow \infty} \frac{1}{N-M} \sum_{n=M}^{N-1} \prod_{i=1}^k T_i^n u_i = 0\]
 in $L^2(X,\mu )$. 
Thus we have that 
\begin{align*}
\liminf_{N - M \rightarrow \infty} \frac{1}{N - M} \sum_{n=M}^{N-1} \mu(A \cap T_1^{-n} A \cap \cdots \cap T_k^{-n} A) 
& = \liminf_{N - M \rightarrow \infty} \frac{1}{N - M} \sum_{n=M}^{N-1} \left\langle \mathbbm{1} _A, \prod_{i=1}^{k} T_i^n f_i\right\rangle _{L^2(X,\mu)} \\
&= \liminf_{N - M \rightarrow \infty} \frac{1}{N - M} \sum_{n=M}^{N-1} \int_A \prod_{i=1}^{k} T_i^n f_i\, d \mu . 
\end{align*}

 Note that for each $i\in \{ 1,\dots ,k\}$, we have $f_i(x)>0$ for $\mu $-almost every $x\in A$ by Lemma \ref{lem:add}, so $\int _A\prod _{i=1}^kf_i\, d\mu > 0$. 
Set $c_0= \frac{1}{2} \int _A\prod _{i=1}^kf_i\, d\mu >0$, and take any $\varepsilon \in (0, c_0)$. 
Given $\delta \in (0,\varepsilon /k)$, one can show that the set
\[ S = \left\{ n\in \mathbb{N} \colon \| T_i^n f_i - f_i \|_2 < \delta \ \text{ for all } i=1,\dots ,k\right\} \]
is syndetic by using Remark \ref{rem:add}. 
To see this, note first that for each $i\in \{ 1,\dots ,k\}$, one can find $F_i\in L^2(X,\mu )$ of the form $F_i=\sum_{j=1}^{l_i} a_{i,j}g_{i,j}$ for some $l_i\in \mathbb{N}$, some $a_{i,1},\dots ,a_{i,l_i}\in \mathbb{C} $, and some (normalized) eigenfunctions $g_{i,1}, \cdots , g_{i,l_i}\in L^2(X,\mu )$ corresponding to eigenvalues $\lambda _{i,1},\dots ,\lambda _{i,l_i}\in \{ z\in \mathbb{C} \colon |z|=1\}$ such that $\left\| f_i - F_i\right\| _2 <\delta /3$. 
Let $a=\max _{1\leq i\leq k}\sum _{j=1}^{l_i}|a_{i,j}|$. 
Then the following set 
\[ E=\left\{ n\in \mathbb{N} \colon |\lambda_{i,j}^n - 1|<\delta /(3a) \ \text{ for all } i=1,\dots ,k \text{ and } j=1,\dots ,l_i \right\} \]
is contained in $S$. 
Indeed, for every $n\in E$, one has 
\begin{align*}
\| T_i^n f_i - f_i \|_2
&\leq \| T_i^n f_i - T_i^nF_i \|_2+\| T_i^n F_i - F_i \|_2+\| F_i - f_i \|_2 \\
&=\| f_i-F_i \|_2+\| T_i^n F_i - F_i \|_2+\| F_i - f_i \|_2<\frac{\delta}{3} +\frac{\delta}{3} +\frac{\delta}{3} =\delta 
\end{align*}
as 
\begin{align*}
\| T_i^n F_i - F_i \|_2&=\left\| \sum_{j=1}^{l_i} a_{i,j}(\lambda _{i,j}^n-1)g_{i,j}\right\| _2\leq \sum_{j=1}^{l_i} |a_{i,j}||\lambda _{i,j}^n-1|\| g_{i,j}\| _2<\frac{\delta}{3} 
\end{align*}
for every $i\in \{ 1,\dots ,k\}$.
Since $E$ is syndetic, so is $S$. 

For $n\in S$, we have by Lemma \ref{lem:add} and the Schwarz inequality that
\begin{align*}
\left| \int _A\prod _{i=1}^kT_i^n f_i\, d\mu -\int _A\prod _{i=1}^kf_i\, d\mu \right|
&\leq 
\int _A\left| \prod _{i=1}^kT_i^n f_i-f_1\prod _{i=2}^kT_i^nf_i\right|\, d\mu +\int _A\left| f_1\prod _{i=2}^kT_i^n f_i-f_1\cdot f_2\prod _{i=3}^kT_i^nf_i\right|\, d\mu \\
&\quad +\cdots +\int _A\left| \left( \prod _{i=1}^{k-1}f_i\right) \cdot T_k^nf_k -\prod _{i=1}^kf_i\right|\, d\mu \\
&\leq \int _A\left| T_1^n f_1-f_1\right| \prod _{i=2}^k\left| T_i^nf_i\right|\, d\mu +\int _A|f_1| \left| T_2^n f_2-f_2\right| \prod _{i=3}^k\left| T_i^nf_i\right|\, d\mu \\
&\quad +\cdots +\int _A\prod _{i=1}^{k-1}\left| f_i\right| \cdot \left| T_k^nf_k -f_k\right|\, d\mu \\
&\leq \int _A\left| T_1^n f_1-f_1\right| \, d\mu +\int _A\left| T_2^n f_2-f_2\right| \, d\mu +\cdots +\int _A \left| T_k^nf_k -f_k\right|\, d\mu \\
&\leq \left\| T_1^n f_1-f_1\right\| _2+\left\| T_2^n f_2-f_2\right\| _2+\cdots +\left\| T_k^nf_k -f_k\right\| _2<k\delta <\varepsilon .
\end{align*}
Thus for $n \in S$, we have
\[ \int_A \prod _{i=1}^kT_i^n f_i\, d\mu \geq \int _A\prod _{i=1}^kf_i \, d \mu - \varepsilon \geq c_0.\]
Hence 
\[ \liminf_{N - M \rightarrow \infty} \frac{1}{N -M} \sum_{n=M}^{N-1} \mu(A \cap T_1^{-n} A \cap \cdots \cap T_k^{-n} A) \geq {d}_{*}(S) \cdot c_0> 0,\]
where $d_{*}(S)= \liminf_{N-M \rightarrow \infty} \frac{1}{N-M} \# \{ n\in S \cap \{M, M+1, \cdots, N-1\} \}$. 
By a similar argument of the proof of Proposition~\ref{functional_kmr_disjoint_general},
%As we did in the proof of Proposition \ref{functional_kmr_disjoint_general}, 
we can conclude the proof.
\end{proof}

\begin{proof}[Proof of Theorem \ref{metric_recurrence:new}] 
Let $\Delta \subset X^{k}$ be the set of diagonal points of $X^{k}$. 
Define $\iota \colon X \to \Delta $ by $x\mapsto \bar{x} = (x,\dots ,x)\in \Delta $, and $\nu =\iota _*\mu $. 
Set $F=T_1 \times \dots \times T_k\colon X^{k}\to X^{k}$, and $D\left( (x_i),(y_i)\right) =\max _{1 \leq i \le k}d(x_i,y_i)$ as a metric on $X^{k}$. 
It is enough to show that there exists a subset $Y\subset \Delta $ with $\nu (\Delta \setminus Y)=0$ such that 
\[
\liminf _{n\to \infty }D\left( \bar{x} ,F^n(\bar{x} )\right) =0
\]
for every $\bar{x} \in Y$. 

For a given $n\in \mathbb{N}$ consider a cover $\mathcal{B} =\{ B(x,1/n)\subset X\colon x\in X \} $ of $X$, where $B(x,r)=\{ y\in X\colon d(x,y)<r\} $.
We can choose a countable subcover $\mathcal{C} =\{ B(x_j,1/n)\in \mathcal{B} \} _{j\in \mathbb{N} }$ of $X$ since a separable metric space is Lindel\"of, i.e., every open cover has a countable subcover.
Thus it holds that
\[
\Delta =\bigcup _{j\in \mathbb{N} }\iota \left( B\left( x_j,1/n\right) \right) .
\]
For notational simplicity, we set $\Delta (x,r)=\iota \left( B(x,r)\right) \subset \Delta $ in the rest of the proof. 
For each $n\in \mathbb{N} $, let
\[
\Delta _n=\bigcup _{j\in \mathbb{N} }\left( \Delta \left( x_j,1/n\right) \setminus \bigcap _{l\in \mathbb{N} }\bigcup _{i\geq l}F^{-i}\left( B\left( x_j,1/n\right) ^{k}\right) \right) \subset \Delta ,
\]
where $B^{k}=B\times \dots \times B\subset X^{k}$. 
Here we have $\nu (\Delta _n)=0$ for every $n\in \mathbb{N}$. 
To see this, given $B\subset X$, let 
\[
B^*=\iota (B)\setminus \bigcap _{l\in \mathbb{N} }\bigcup _{i\geq l}F^{-i}(B^k) 
= \bigcup_{l \in \mathbb{N}} B_l^*,
\]
where $B_l^* = \iota(B) \bigcap \left( \bigcap_{i \geq l} (X^k \setminus F^{-i}(B^k)) \right)$. Let $B_l = \iota^{-1} (B_l^*)$. 
Note that 
\begin{align*}
B_l &= B \bigcap  \left(\bigcap_{i \geq l} \iota^{-1} (X^k \setminus F^{-i}(B^k)) \right)
      = B \bigcap \left( \bigcap_{i \geq l} X \setminus (T_1^{-i} B \cap \cdots \cap T_k^{-i} B)\right)  \\
      &= \bigcap_{i \geq l} \left( B \setminus (T_1^{-i} B \cap \cdots \cap T_k^{-i} B) \right).
\end{align*}
Thus, for $n \geq l$, $B_l \cap T_1^{-n} B_l \cap \cdots \cap T_k^{-n} B_l = \emptyset.$
Then by Theorem \ref{recurrence:new} for mutually disjoint transformations $T_1,\dots ,T_k$ on $(X,\mu )$ one has $\mu (B_l)=0$, and so $\nu (B_l^*)=0$. Therefore $\nu(B^*) =0$. 
It follows that 
\[
\nu \left( \Delta \left( x_j,1/n\right) \setminus \bigcap _{l\in \mathbb{N} }\bigcup _{i\geq l}F^{-i}\left( B\left( x_j,1/n\right) ^{k} \right) \right) =0
\]
for every $j\in \mathbb{N}$, and thus $\nu (\Delta _n)=0$ for every $n\in \mathbb{N}$. 
For $Y=\Delta \setminus \cup _{n\in \mathbb{N} }\Delta _n$, the result follows. 
\end{proof}
\section*{Acknowledgments} 
The authors wish to thank the anonymous referee for the valuable comments and suggestions. 
MH was supported by Japan Society for the Promotion of Science (JSPS) KAKENHI Grant Number 19K03558. 
DK was supported by the National Research Foundation of Korea (NRF-2018R1A2B6001624). 
YS was supported by the National Research Foundation of Korea (NRF-2020R1A2C1A01005446).


\begin{bibdiv}
\begin{biblist}

\bib{Austin2010}{article}{
   author={Austin, Tim},
   title={On the norm convergence of non-conventional ergodic averages},
   journal={Ergodic Theory Dynam. Systems},
   volume={30},
   date={2010},
   number={2},
   pages={321--338},
   issn={0143-3857},
   review={\MR{2599882}},
   doi={10.1017/S014338570900011X},
}

\bib{Berend1985}{article}{
   author={Berend, Daniel},
   title={Joint ergodicity and mixing},
   journal={J. Analyse Math.},
   volume={45},
   date={1985},
   pages={255--284},
   issn={0021-7670},
   review={\MR{833414}},
   doi={10.1007/BF02792552},
}

\bib{Berend-Bergelson1984}{article}{
   author={Berend, Daniel},
   author={Bergelson, Vitaly},
   title={Jointly ergodic measure-preserving transformations},
   journal={Israel J. Math.},
   volume={49},
   date={1984},
   number={4},
   pages={307--314},
   issn={0021-2172},
   review={\MR{788255}},
   doi={10.1007/BF02760955},
}

\bib{Berend-Bergelson1986}{article}{
   author={Berend, Daniel},
   author={Bergelson, Vitaly},
   title={Characterization of joint ergodicity for noncommuting
   transformations},
   journal={Israel J. Math.},
   volume={56},
   date={1986},
   number={1},
   pages={123--128},
   issn={0021-2172},
   review={\MR{879919}},
   doi={10.1007/BF02776245},
}

\bib{Bergelson-Host-Kra2005}{article}{
   author={Bergelson, Vitaly},
   author={Host, Bernard},
   author={Kra, Bryna},
   title={Multiple recurrence and nilsequences},
   note={With an appendix by Imre Ruzsa},
   journal={Invent. Math.},
   volume={160},
   date={2005},
   number={2},
   pages={261--303},
   issn={0020-9910},
   review={\MR{2138068}},
   doi={10.1007/s00222-004-0428-6},
}

\bib{Bergelson-Leibman2002}{article}{
   author={Bergelson, V.},
   author={Leibman, A.},
   title={A nilpotent Roth theorem},
   journal={Invent. Math.},
   volume={147},
   date={2002},
   number={2},
   pages={429--470},
   issn={0020-9910},
   review={\MR{1881925}},
   doi={10.1007/s002220100179},
}

\bib{Chu2011}{article}{
   author={Chu, Qing},
   title={Multiple recurrence for two commuting transformations},
   journal={Ergodic Theory Dynam. Systems},
   volume={31},
   date={2011},
   number={3},
   pages={771--792},
   issn={0143-3857},
   review={\MR{2794947}},
   doi={10.1017/S0143385710000258},
}

\bib{Donoso-Sun2018}{article}{
   author={Donoso, Sebasti\'{a}n},
   author={Sun, Wenbo},
   title={Quantitative multiple recurrence for two and three
   transformations},
   journal={Israel J. Math.},
   volume={226},
   date={2018},
   number={1},
   pages={71--85},
   issn={0021-2172},
   review={\MR{3819687}},
   doi={10.1007/s11856-018-1690-4},
}

\bib{Furstenberg1967}{article}{
   author={Furstenberg, Harry},
   title={Disjointness in ergodic theory, minimal sets, and a problem in
   Diophantine approximation},
   journal={Math. Systems Theory},
   volume={1},
   date={1967},
   pages={1--49},
   issn={0025-5661},
   review={\MR{0213508}},
   doi={10.1007/BF01692494},
}

\bib{Furstenberg1977}{article}{
   author={Furstenberg, Harry},
   title={Ergodic behavior of diagonal measures and a theorem of Szemer\'{e}di
   on arithmetic progressions},
   journal={J. Analyse Math.},
   volume={31},
   date={1977},
   pages={204--256},
   issn={0021-7670},
   review={\MR{0498471}},
   doi={10.1007/BF02813304},
}

\bib{Furstenberg1981}{book}{
   author={Furstenberg, H.},
   title={Recurrence in ergodic theory and combinatorial number theory},
   note={M. B. Porter Lectures},
   publisher={Princeton University Press, Princeton, N.J.},
   date={1981},
   pages={xi+203},
   isbn={0-691-08269-3},
   review={\MR{603625}},
}

\bib{Furstenberg-Katznelson1978}{article}{
   author={Furstenberg, H.},
   author={Katznelson, Y.},
   title={An ergodic Szemer\'{e}di theorem for commuting transformations},
   journal={J. Analyse Math.},
   volume={34},
   date={1978},
   pages={275--291 (1979)},
   issn={0021-7670},
   review={\MR{531279}},
   doi={10.1007/BF02790016},
}

\bib{Furstenberg-Katznelson-Ornstein1982}{article}{
   author={Furstenberg, H.},
   author={Katznelson, Y.},
   author={Ornstein, D.},
   title={The ergodic theoretical proof of Szemer\'{e}di's theorem},
   journal={Bull. Amer. Math. Soc. (N.S.)},
   volume={7},
   date={1982},
   number={3},
   pages={527--552},
   issn={0273-0979},
   review={\MR{670131}},
   doi={10.1090/S0273-0979-1982-15052-2},
}

\bib{Hirayama2019}{article}{
   author={Hirayama, Michihiro},
   title={Bounds for multiple recurrence rate and dimension},
   journal={Tokyo J. Math.},
   volume={42},
   date={2019},
   number={1},
   pages={239--253}, 
   issn={0387-3870},
   review={\MR{3982057}},
}

\bib{Host2009}{article}{
   author={Host, Bernard},
   title={Ergodic seminorms for commuting transformations and applications},
   journal={Studia Math.},
   volume={195},
   date={2009},
   number={1},
   pages={31--49},
   issn={0039-3223},
   review={\MR{2539560}},
   doi={10.4064/sm195-1-3},
}

\bib{Host-Kra2005}{article}{
   author={Host, Bernard},
   author={Kra, Bryna},
   title={Nonconventional ergodic averages and nilmanifolds},
   journal={Ann. of Math. (2)},
   volume={161},
   date={2005},
   number={1},
   pages={397--488},
   issn={0003-486X},
   review={\MR{2150389}},
   doi={10.4007/annals.2005.161.397},
}

\bib{Khintchine1935}{article}{
   author={Khintchine, A.},
   title={Eine Versch\"{a}rfung des Poincar\'{e}schen ``Wiederkehrsatzes''},
   language={German},
   journal={Compositio Math.},
   volume={1},
   date={1935},
   pages={177--179},
   issn={0010-437X},
   review={\MR{1556883}},
}

\bib{Kim2009}{article}{
   author={Kim, Dong Han},
   title={Quantitative recurrence properties for group actions},
   journal={Nonlinearity},
   volume={22},
   date={2009},
   number={1},
   pages={1--9},
   issn={0951-7715},
   review={\MR{2470261}},
   doi={10.1088/0951-7715/22/1/001},
}

\bib{Krengel1985}{book}{
   author={Krengel, Ulrich},
   title={Ergodic theorems},
   series={De Gruyter Studies in Mathematics},
   volume={6},
   note={With a supplement by Antoine Brunel},
   publisher={Walter de Gruyter \& Co., Berlin},
   date={1985},
   pages={viii+357},
   isbn={3-11-008478-3},
   review={\MR{797411}},
}

\bib{Lesigne-Rittaud-delaRue2003}{article}{
   author={Lesigne, E.},
   author={Rittaud, B.},
   author={de la Rue, T.},
   title={Weak disjointness of measure-preserving dynamical systems},
   journal={Ergodic Theory Dynam. Systems},
   volume={23},
   date={2003},
   number={4},
   pages={1173--1198},
   issn={0143-3857},
   review={\MR{1997972}},
   doi={10.1017/S0143385702001505},
}

\bib{Rudin1987}{book}{
   author={Rudin, Walter},
   title={Real and complex analysis},
   edition={3},
   publisher={McGraw-Hill Book Co., New York},
   date={1987},
   pages={xiv+416},
   isbn={0-07-054234-1},
   review={\MR{924157}},
}

\bib{Rudolph1990}{book}{
   author={Rudolph, Daniel J.},
   title={Fundamentals of measurable dynamics},
   series={Oxford Science Publications},
   note={Ergodic theory on Lebesgue spaces},
   publisher={The Clarendon Press, Oxford University Press, New York},
   date={1990},
   pages={x+168},
   isbn={0-19-853572-4},
   review={\MR{1086631}},
}

\bib{delaRue2012}{article}{
   author={de la Rue, Thierry},
   title={Joinings in ergodic theory},
   conference={
      title={Mathematics of complexity and dynamical systems. Vols. 1--3},
   },
   book={
      publisher={Springer, New York},
   },
   date={2012},
   pages={796--809},
   review={\MR{3220709}},
%  doi={10.1007/978-1-4614-1806-1_49},
}

\bib{Tao2008}{article}{
   author={Tao, Terence},
   title={Norm convergence of multiple ergodic averages for commuting
   transformations},
   journal={Ergodic Theory Dynam. Systems},
   volume={28},
   date={2008},
   number={2},
   pages={657--688},
   issn={0143-3857},
   review={\MR{2408398}},
   doi={10.1017/S0143385708000011},
}

\bib{Towsner2009}{article}{
   author={Towsner, Henry},
   title={Convergence of diagonal ergodic averages},
   journal={Ergodic Theory Dynam. Systems},
   volume={29},
   date={2009},
   number={4},
   pages={1309--1326},
   issn={0143-3857},
   review={\MR{2529651}},
   doi={10.1017/S0143385708000722},
}

\bib{Walsh2012}{article}{
   author={Walsh, Miguel N.},
   title={Norm convergence of nilpotent ergodic averages},
   journal={Ann. of Math. (2)},
   volume={175},
   date={2012},
   number={3},
   pages={1667--1688},
   issn={0003-486X},
   review={\MR{2912715}},
   doi={10.4007/annals.2012.175.3.15},
}

\bib{Walters1982}{book}{
   author={Walters, Peter},
   title={An introduction to ergodic theory},
   series={Graduate Texts in Mathematics},
   volume={79},
   publisher={Springer-Verlag, New York-Berlin},
   date={1982},
   pages={ix+250},
   isbn={0-387-90599-5},
   review={\MR{648108}},
}

\bib{Ziegler2007}{article}{
   author={Ziegler, Tamar},
   title={Universal characteristic factors and Furstenberg averages},
   journal={J. Amer. Math. Soc.},
   volume={20},
   date={2007},
   number={1},
   pages={53--97},
   issn={0894-0347},
   review={\MR{2257397}},
   doi={10.1090/S0894-0347-06-00532-7},
}

\end{biblist}
\end{bibdiv}
\end{document}